\documentclass[12pt, a4paper]{article}
\usepackage{graphicx}
\usepackage{amssymb}
\usepackage{stmaryrd}
\usepackage{enumitem}

\usepackage{geometry}
\usepackage{tikz}
\usepackage{float}
\usepackage{url}
\usepackage{mathtools}
\usepackage{amsmath,amsfonts,amsthm}
\usepackage{mathrsfs}

\newtheorem{theorem}{Theorem}[section]
\newtheorem{lemma}[theorem]{Lemma}
\newtheorem{proposition}{Proposition}[section]
\newtheorem{corollary}[theorem]{Corollary}
\newtheorem{definition}{Definition}

\usepackage{pgfplots}
\pgfplotsset{compat=1.18}
\newcommand{\SO}{\operatorname{SO}}
\DeclareMathOperator{\TC}{\mathscr{C}_{T}}
\DeclareMathOperator{\PP}{\mathcal{P}}
\usepackage{tikz}
\usetikzlibrary{positioning}
\usetikzlibrary{arrows.meta}

\usepackage{subfig}

\usepackage{ytableau}
\usepackage[colorlinks,
linkcolor=blue,
anchorcolor=blue,
citecolor=blue
]{hyperref}
\begin{document}
\begin{center}
{\large \bf  Structural Components Dominate Asymptotic Behavior on Sombor Index with Iterated Pendant Constructions}
\end{center}
\begin{center}
 Jasem Hamoud \\[6pt]
Physics and Technology School of Applied Mathematics and Informatics \\
Moscow Institute of Physics and Technology, 141701, Moscow region, Russia\\[6pt]
Email: {\tt hamoud.math@gmail.com}
\end{center}
\noindent
\begin{abstract}
The Sombor index, a degree-based topological descriptor introduced by Gutman in 2021, lacks closed-form expressions for complex hierarchical trees with multi-level pendant structures and nonuniform degree distributions, despite extensive results for simpler families such as paths, stars, cycles, and basic caterpillars. For a simple graph $\mathcal{G}$, the Sombor index is defined as
\[
\mathrm{SO}(\mathcal{G}) = \sum_{uv \in E(\mathcal{G})} \sqrt{d(v)^2 + d(u)^2}.
\]

In this work, we derive a general recursive formula for the Sombor index of multi-level pendant-augmented path trees. These trees are constructed from a spine path $\mathcal{P}_n$ ($n \ge 2$) in which each vertex has degree $2+k$ and are iteratively augmented over $m \ge 1$ hierarchical levels. Pendants attached to odd-indexed spine vertices branch with replication factor $k$ and terminal degree $\ell_i$, whereas those stemming from even-indexed vertices incorporate an initial offset $\ell_1>2$ that propagates through subsequent levels.
These results significantly advance the theoretical and computational study of degree-based topological descriptors in iteratively constructed graphs.
\end{abstract}

\noindent\rule{17cm}{1.0pt}

\noindent\textbf{AMS Classification 2010:} 05C05, 05C12, 05C20, 05C25, 05C35, 05C76, 68R10.


\noindent\rule{17cm}{1.0pt}

\section{Introduction}
Throughout this paper, let $\mathcal{G}=(V,E)$ be a simple connected graph with vertex set $V(\mathcal{G})=\{v_1,v_2,\dots,v_n\}$, where $n=|V(\mathcal{G})|$, and edge set $E(\mathcal{G})=\{e_1,e_2,\dots,e_m\}$, where $m=|E(\mathcal{G})|$.  
Let $\mathscr{D}=(d_1,d_2,\dots,d_n)$ be the degree sequence of $\mathcal{G}$, where $d_w=d_\mathcal{G}(v_w)$ is the degree of vertex $v_w \in V(\mathcal{G})$ for each $w$, and assume $\mathscr{D}$ is arranged in increasing order. 
S. L. Hakimi~\cite{Hakimi1962} introduced the concept of realizability for a sequence of nonnegative integers as the degree sequence of a simple graph.  
Caro and Pepper~\cite{Caro2014Pepper} presented a study of the degree sequence index strategy.

The Wiener index $W(\mathcal{G}) = \sum_{u \sim v} d(u,v)$ of a connected graph $G$ is a well-studied topological index in mathematical and chemical graph theory; see the reviews~\cite{reef1,reef4,reef6,reef10} for details. Gutman~\cite{Gutman2021Geo} established the function $F(x,y)=\sqrt{x^2+y^2}$ as a novel graph-theoretic descriptor and defined the Sombor index accordingly
\begin{equation}~\label{eqq1So}
\SO(\mathcal{G})=\sum_{uv\in E(\mathcal{G})} \sqrt{d(v)^2+d(u)^2}.
\end{equation}
If two unicyclic graphs $\mathcal{G}_1$ and $\mathcal{G}_2$ are isomorphic, then $\SO(\mathcal{G}_1) = \SO(\mathcal{G}_2)$. Indeed, suppose $\mathcal{G}_1 \cong \mathcal{G}_2$. Then there exists a bijection $\varphi \colon V(\mathcal{G}_1) \to V(\mathcal{G}_2)$ that preserves adjacency: $uv \in E(\mathcal{G}_1)$ if and only if $\varphi(u)\varphi(v) \in E(\mathcal{G}_2)$. This bijection necessarily preserves vertex degrees, so $d_{\mathcal{G}_1}(v)=d_{\mathcal{G}_2}(\varphi(v))$ for every $v \in V(\mathcal{G}_1)$.

Various parameters have been studied to match topological indices with different graph classes. One study examined the Sombor index for trees and unicyclic graphs with a given maximum degree~\cite{ZhouMiao}. Zhang, Meng, and Wang~\cite{ZhangMeng} extended this research to several graph families, particularly determining sharp upper bounds on the Sombor index for connected bipartite graphs on $n$ vertices with a fixed matching number. Recently, Zhou,~T.et al~\cite{ZhouMiaon2} characterized the trees on $n$ vertices that maximize the Sombor index among those with a prescribed matching independence number.

Several classical indices, such as the Zagreb indices, are also used to quantify structural features of graphs, particularly in chemical graph theory.
\begin{definition}[Zagreb Indices~\cite{gutman1972total, gutman1975acyclic}]
Let $\mathcal{G}=(V, E) $ be a graph. The first and second Zagreb indices are given by
\[
M_1(\mathcal{G}) = \sum_{v \in V(\mathcal{G})} d_\mathcal{G}(v)^2, \quad M_2(\mathcal{G}) = \sum_{uv \in E(G)} d_\mathcal{G}(u)d_\mathcal{G}(v).
\]
\end{definition}

\begin{lemma}[\cite{Das2021}]~\label{perlemnumb1}
Let $\mathcal{G}=(V,E)$ be a connected graph. Then
\begin{enumerate}
\item if $uv \in E(\mathcal{G})$, then $\SO(\mathcal{G})>\SO(\mathcal{G}-uv)$;
\item if $uv \notin E(\mathcal{G})$, then $\SO(\mathcal{G}+uv)>\SO(\mathcal{G})$.
\end{enumerate}
\end{lemma}

Since $\varphi$ induces a bijection between the edge sets $E(\mathcal{G}_1)$ and $E(\mathcal{G}_2)$, and the endpoints of corresponding edges have identical degree pairs, each edge $uv$ in $\mathcal{G}_1$ contributes exactly the same value $\sqrt{d(u)^2 + d(v)^2}$ as its image $\varphi(u)\varphi(v)$ in $\mathcal{G}_2$. 
Thus, the Sombor index is invariant under graph isomorphism. The converse, however, does not hold in general: two non-isomorphic unicyclic graphs may have the same Sombor index. This is because the index depends only on the multiset of degree pairs $\{(d(u), d(v)) : uv \in E(\mathcal{G})\}$ and carries no further information about the global structure of the graph. This property underscores both the strength and the intrinsic limitation of the Sombor index as a graph invariant.

Several important results concerning the Sombor index that are essential for understanding the findings presented later.

\begin{theorem}[\cite{Gutman2021Geo}]
Let $K_n$ denote the complete graph on $n$ vertices and let $\overline{K_n}$ denote its complement (the empty graph on $n$ vertices). Then for every graph $G$ of order $n$,
$$
\SO(\overline{K_n}) \leq \SO(\mathcal{G}) \leq \SO(K_n),
$$
with equality if and only if $\mathcal{G} \cong \overline{K_n}$ or $\mathcal{G} \cong K_n$.  
\end{theorem}

\begin{theorem}[\cite{Gutman2021Geo,Liu2023GutmanYou}]
Let $\mathcal{P}_n$ denote the path on $n$ vertices. Then for every connected graph $\mathcal{G}$ of order $n$,
$$
\SO(\mathcal{P}_n) \leq \SO(\mathcal{G}) \leq \SO(K_n),
$$
with equality if and only if $\mathcal{G} \cong \mathcal{P}_n$ or $\mathcal{G} \cong K_n$. For $n \geq 3$,
$\SO(\mathcal{P}_n)=2\sqrt{5} + 2(n-3)\sqrt{2}$.
\end{theorem}
In this work, we derive a closed-form formula for the Sombor index of a class of multi-level uniformly extended caterpillar trees $C(n,p,k,\ell_1,\ell_2,\dots,\ell_m)$. While the Sombor index has been studied extensively for various trees and chemical graphs, an explicit expression for this general parametric family has not appeared in the literature. Our formula provides an exact computation for any number of spine vertices, initial leaves, and pendant extensions, thereby generalizing previously known results for simple caterpillars and offering a foundation for further extremal and structural studies of the Sombor index. 

\subsection{Problem of search}

A significant research gap in the study of the Sombor index lies in the limited availability of closed-form expressions for complex hierarchical graph structures, particularly trees with multi-level pendant constructions and nonuniform degree distributions. While the Sombor index has been extensively investigated for classical graph families such as paths, stars, cycles, chemical trees, and simple caterpillars, most existing results focus on graphs with relatively simple degree patterns or a single layer of pendant vertices. Theoretical developments often emphasize extremal bounds, comparative inequalities, or computational evaluations for restricted families. However, there is a noticeable absence of general analytical frameworks capable of handling iterated or recursively constructed trees in which degree patterns propagate across multiple levels. This limitation restricts the deeper structural understanding of how the Sombor index behaves under systematic degree augmentation.

Another important gap concerns the theoretical analysis of degree propagation effects in topological indices. The Sombor index is inherently degree-based, depending on the squared degrees of adjacent vertices. Despite this, much of the literature treats vertex degrees as static local parameters rather than dynamic quantities that evolve through hierarchical construction. When pendant vertices themselves receive additional attachments, their degrees increase and begin influencing higher-level contributions in nontrivial ways. The cumulative impact of such multi-level degree interactions on the global Sombor index remains largely unexplored. Without explicit recursive formulations, it is difficult to determine how structural growth influences index escalation, whether growth is linear, polynomial, or exponential with respect to depth, and which structural components dominate asymptotic behavior. Addressing this gap requires systematic formulas that capture how contributions from one level influence subsequent levels, thereby transforming isolated calculations into a coherent structural theory.

\section{Exact Sombor Index of Path Graphs with Alternating and Heterogeneous Pendant Structures}
In this section, throughout the Proposition~\ref{SO001caterpillar}, we formalize the structure of the hierarchical caterpillar-type tree $\TC$ under consideration. The tree $\TC$ is rooted and constructed from a spine consisting of a path on $n$ vertices at level $0$. To each spine vertex, we attach $p$ vertices at level $1$. Recursively, each vertex at level $1$ connects to $k$ vertices at level $2$, and for $i \ge 2$, each vertex at level $i$ connects to $\ell_i$ vertices at level $i+1$. This iterative branching yields a multi-level hierarchical tree with controlled degree distribution at each level. Since the Sombor index depends solely on the degrees of adjacent vertices, the layered structure enables systematic computation by partitioning the edge set across levels. The following lemma provides an explicit closed-form expression for the Sombor index of $\TC$, obtained by summing contributions from edges between consecutive levels.

\begin{proposition}~\label{SO001caterpillar}
Let $\TC(n,p,k,\ell_1,\ell_2,\dots,\ell_m)$ be the caterpillar graph obtained from the spine cycle $C(n,p)$ by applying the iterated uniform pendant extension process described above, where $n\ge 2$, $p\ge 0$, and $m\ge 0$. Then its Sombor index is given by
\begin{equation}\label{eqq1SO001caterpillar}
\begin{aligned}
\SO\bigl(C(n,p,k,\ell_1,\dots,\ell_m)\bigr)
&= (n-1)\sqrt{2}\,(p+2) \\
&\quad + np\,\sqrt{(p+2)^2 + (k+1)^2} \\
&\quad + npk\,\sqrt{(k+1)^2 + (\ell_1+1)^2} \\
&\quad + npk\ell_1\,\sqrt{(\ell_1+1)^2 + (\ell_2+1)^2} \\
&\quad + \cdots \\
&\quad + npk\ell_1\cdots\ell_{m-1}\,\sqrt{(\ell_{m-1}+1)^2 + (\ell_m+1)^2} \\
&\quad + npk\ell_1\cdots\ell_m\,\sqrt{(\ell_m+1)^2 + 1}.
\end{aligned}
\end{equation}
\end{proposition}

\begin{proof}
We compute the Sombor index by partitioning the edge set according to the levels of the attached substructures. 
The spine is a path on $n$ vertices and thus contains $n-1$ edges. Each internal vertex of the spine has degree $p+2$, while each endpoint has degree $p+1$. For $n \ge 2$, every spine edge is incident to at least one internal vertex. Hence, all spine edges connect vertices of degree $p+2$ and each contributes
\begin{equation}\label{eqq2SO001caterpillar}
\sqrt{(p+2)^2 + (p+2)^2} = \sqrt{2}\,(p+2)
\end{equation}
to the Sombor index. The total contribution from the $n-1$ spine edges is therefore
\begin{equation}\label{eqq3SO001caterpillar}
(n-1)\sqrt{2}\,(p+2).
\end{equation}

Note that the edges between level 0 and level 1, each of the $n$ vertices in level 0 (the spine) is adjacent to exactly $p$ vertices in level 1, giving a total of $np$ edges between these levels. Vertices in level 0 have degree $p+2$, while each vertex in level 1 is adjacent to exactly one vertex in level 0 and to $k$ vertices in level 2, and thus has degree $k+1$. Each such edge contributes $\sqrt{(p+2)^2 + (k+1)^2}$
to the Sombor index. Thus, from~\eqref{eqq2SO001caterpillar} and \eqref{eqq3SO001caterpillar} we obtain
\begin{equation}~\label{eqq4SO001caterpillar}
\lambda_1=np \sqrt{(p+2)^2 + (k+1)^2}.
\end{equation}
Therefore,  from~\eqref{eqq4SO001caterpillar} we obtain 
\begin{equation}~\label{eqq5SO001caterpillar}
\SO_{1}(\TC)=\sqrt{(p+2)^2 + (p+2)^2} = \sqrt{2}\,(p+2)+(n-1)\sqrt{2}\,(p+2)+\lambda_1.
\end{equation}
Similarly, for the edges between level 1 and level 2. Note that each of the $np$ vertices in level 1 is adjacent to exactly $k$ vertices in level 2, yielding $npk$ edges between these levels. Vertices in level 1 have degree $k+1$, while each vertex in level 2 is adjacent to exactly one vertex in level 1 and to $\ell_1$ vertices in level 3, and thus has degree $\ell_1+1$. Thus, 
\begin{equation}~\label{eqq6SO001caterpillar}
\SO_{2}(\TC)=\sqrt{(k+1)^2 + (\ell_1+1)^2}+npk \sqrt{(k+1)^2 + (\ell_1+1)^2}.
\end{equation}

According to~\eqref{eqq5SO001caterpillar} and \eqref{eqq6SO001caterpillar}, by proceeding inductively, for each $i=1,\dots,m-1$, there are $npk\ell_1\cdots\ell_{i-1}$ edges between level $i$ and level $i+1$, each contributing $\sqrt{(\ell_{i-1}+1)^2 + (\ell_i+1)^2}$. Finally, there are $npk\ell_1\cdots\ell_m$ edges between level $m$ and level $m+1$ (the leaves), each contributing $\sqrt{(\ell_m+1)^2 + 1}.$ Thus, the relationship~\eqref{eqq1SO001caterpillar} holds, which completes the proof.
\end{proof}


We explicitly derive the Sombor index of a path with alternating vertex degrees, as presented in Proposition~\ref{SO002caterpillar}.
\begin{proposition}~\label{SO002caterpillar}
Let $\PP_n=v_1-v_2-\dots-v_n$ be a path on $n \geqslant 2$ vertices whose degrees alternate as
\[
d(v_i) =
\begin{cases}
2p, & i \text{ odd}, \\
2p+1, & i \text{ even},
\end{cases}
\]
for some integer $p \geqslant 1$. Then the Sombor index of $\PP_n$ is
\begin{equation}~\label{eqq1SO002caterpillar}
\SO(\PP_n)=(n-1)\sqrt{8p^2+4p+1}. 
\end{equation}
\end{proposition}
\begin{proof}
Assume $\PP_n=v_1-v_2-\dots-v_n$ be a path on $n \geqslant 2$ vertices.
We compute the Sombor index of $\PP_n$ by analyzing its structure and vertex degrees.
For a graph $G=\PP_n$, which is a simple path on $n$ vertices, the sum runs over consecutive vertex pairs $(v_i, v_{i+1})$ for $i = 1, \dots, n-1$. Define the edges of $\PP_n$ as $E(\PP_n)=\{ v_1v_2, v_2v_3, \dots, v_{n-1}v_n \}$. Then, 
\begin{itemize}
    \item Edges where $v_i$ has an odd index and $v_{i+1}$ has an even index.
    \item Edges where $v_i$ has an even index and $v_{i+1}$ has an odd index.
\end{itemize}
For an odd-indexed vertex $v_i$, we have the degree is $d(v_i)=2p$,  and for an even-indexed vertex $v_j$, we have $d(v_j)=2p+1$. Then, we have 
\begin{equation}~\label{eqq2SO002caterpillar}
\SO_1(\PP_n)=\sqrt{8p^2 + 4p + 1}.
\end{equation}
Similarly, from~\eqref{eqq2SO002caterpillar} for an edge connecting an even-indexed vertex $v_{2k}$ to
the next odd-indexed vertex $v_{2k+1}$, satisfying $E(\PP_n)=n-1$. Then, the Sombor index is $\SO(\PP_n)=(n-1)\,\SO_1(\PP_n)$. This completes the proof.
\end{proof}

Observe that, although the vertex degrees alternate along any path, the Sombor weight is symmetric with respect to the degrees of its endpoints. Thus, each edge contributes equally to the index irrespective of whether it joins an odd-degree vertex to an even-degree vertex or an even-degree vertex to an odd-degree one (see Figure~\ref{Tree001characterized}). This symmetry substantially simplifies the calculation, resulting in the compact formula presented in Equation~\eqref{eqq1SO002caterpillar}.
\begin{figure}[H]
    \centering
\begin{tikzpicture}[scale=1, 
    vertex/.style={circle, draw, fill=black, inner sep=1.5pt}, 
    label/.style={font=\footnotesize}]
\draw (0,0) node[vertex, label=below:$v_1$] (v1) {};
\draw (2,0) node[vertex, label=below:$v_2$] (v2) {};
\draw (4,0) node[vertex, label=below:$v_3$] (v3) {};
\draw (6,0) node[vertex, label=below:$v_4$] (v4) {};
\draw (8,0) node[vertex, label=below:$v_5$] (v5) {};
\draw (v1) -- (v2);
\draw (v2) -- (v3);
\draw (v3) -- (v4);
\draw (v4) -- (v5);
\draw (v1) -- (-0.5,1) node[vertex, label=above:$v_{1,1}$] {};
\draw (v1) -- (0.5,1) node[vertex, label=above:$v_{1,2}$] {};
\draw (v1) -- (-0.5,-1) node[vertex, label=below:$v_{1,3}$] {};
\draw (v1) -- (0.5,-1) node[vertex, label=below:$v_{1,4}$] {};

\draw (v2) -- (1.5,1) node[vertex, label=above:$v_{2,1}$] {};
\draw (v2) -- (2.5,1) node[vertex, label=above:$v_{2,2}$] {};
\draw (v2) -- (1.5,-1) node[vertex, label=below:$v_{2,3}$] {};
\draw (v2) -- (2.5,-1) node[vertex, label=below:$v_{2,4}$] {};
\draw (v2) -- (2,1) node[vertex, label=above:$v_{2,5}$] {};
\draw (v3) -- (3.5,1) node[vertex, label=above:$v_{3,1}$] {};
\draw (v3) -- (4.5,1) node[vertex, label=above:$v_{3,2}$] {};
\draw (v3) -- (3.5,-1) node[vertex, label=below:$v_{3,3}$] {};
\draw (v3) -- (4.5,-1) node[vertex, label=below:$v_{3,4}$] {};
\draw (v4) -- (5.5,1) node[vertex, label=above:$v_{4,1}$] {};
\draw (v4) -- (6.5,1) node[vertex, label=above:$v_{4,2}$] {};
\draw (v4) -- (5.5,-1) node[vertex, label=below:$v_{4,3}$] {};
\draw (v4) -- (6.5,-1) node[vertex, label=below:$v_{4,4}$] {};
\draw (v4) -- (6,1) node[vertex, label=above:$v_{4,5}$] {};
\draw (v5) -- (7.5,1) node[vertex, label=above:$v_{5,1}$] {};
\draw (v5) -- (8.5,1) node[vertex, label=above:$v_{5,2}$] {};
\draw (v5) -- (7.5,-1) node[vertex, label=below:$v_{5,3}$] {};
\draw (v5) -- (8.5,-1) node[vertex, label=below:$v_{5,4}$] {};
\end{tikzpicture}
    \caption{Caterpillar trees are precisely those trees characterized by Lemma~\ref{SO002caterpillar}.}
    \label{Tree001characterized}
\end{figure}
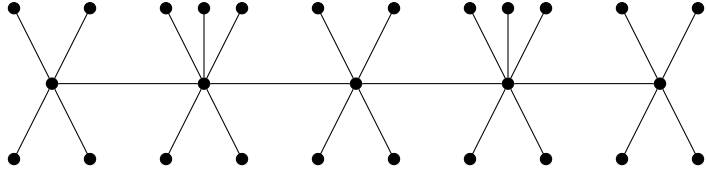

This is where we directly obtain Corollary~\ref{Som01corollary}, as through mathematical induction and symbolic steps, we obtain the following result.
\begin{corollary}~\label{Som01corollary}
Consider the path tree given by Lemma~\ref{SO002caterpillar}, where the degrees of the vertices alternate as
\[
d(v_i) =
\begin{cases}
2p, & i \text{ odd},\\
2p+k, & i \text{ even},
\end{cases}
\]
for integers $p\ge 1$ and $k>1$. Then the Sombor index of $P_n$ is
\begin{equation}\label{eqp1Som01corollary}
\SO(\PP_n)=(n-1)\,\sqrt{8p^2+4pk+k^2}.
\end{equation}
\end{corollary}

\begin{lemma}\label{SO003caterpillar}
Let $\PP_n=v_1-v_2-\dots-v_n$ be a path on $n \geqslant 2$ vertices. The spine vertices alternate in type, with pendant vertices attached as follows:

\begin{itemize}
    \item Each odd-indexed spine vertex $v_{2i-1}$ has degree $d(v_{2i-1}) = 2 + k$ and is adjacent to $k$ pendant vertices, each of degree $k$.
    \item Each even-indexed spine vertex $v_{2i}$ has degree $d(v_{2i}) = 2 + k$ and is adjacent to $k$ pendant vertices, each of degree $k+1$.
\end{itemize}

The Sombor index of this augmented path is then given by
\begin{equation}~\label{eqq1SO003caterpillar}
\SO(\PP_n) = (n-1)\sqrt{2}\,(2+k) 
+ \left\lfloor \frac{n}{2} \right\rfloor k \sqrt{(2+k)^2 + k^2} 
+ \left\lceil \frac{n-1}{2} \right\rceil k \sqrt{(2+k)^2 + (k+1)^2}.
\end{equation}
\end{lemma}

\begin{proof}
Assume the spine edges are $E_\text{spine} = \{ v_1v_2, v_2v_3, \dots, v_{n-1}v_n \}$. Note that each spine edge connects one odd-indexed and one even-indexed vertex where $d(v_\text{odd})=2+k$ and $d(v_\text{even})=2+k$.
Hence, 
\begin{equation}~\label{eqq2SO003caterpillar}
\SO_1(\PP_n)=(n-1)\,\sqrt{2}\,(2+k).
\end{equation}
Since there are $\lfloor n/2 \rfloor$ odd spine vertices, each with $k$ pendants, then
\begin{equation}~\label{eqq3SO003caterpillar}
\SO_2(\PP_n)=\Big\lfloor \frac{n}{2} \Big\rfloor k \sqrt{(2+k)^2 + k^2}.
\end{equation}
Therefore, from~\eqref{eqq2SO003caterpillar} and \eqref{eqq3SO003caterpillar} by considering each even spine vertex has $k$ pendant vertices of degree $k+1$.  
There are $\lceil (n-1)/2 \rceil$ even spine vertices, each with $k$ pendants. Thus, 
\begin{equation}~\label{eqq4SO003caterpillar}
\SO_3(\PP_n)=\Big\lceil \frac{n-1}{2} \Big\rceil k \sqrt{(2+k)^2 + (k+1)^2}.
\end{equation}
Thus, from~\eqref{eqq4SO003caterpillar} we find that 
\[
\SO(\PP_n)=\sum_{i=1}^{3}\SO_i(\PP_n),
\]
which yields formula~\eqref{eqq1SO003caterpillar}.
This completes the proof.
\end{proof}
This general formula~\eqref{eqq1SO003caterpillar} accounts for both spine edges and pendant edges with arbitrary pendant degrees. When all pendant vertices have degree $k=1$, the expression reduces to the previously obtained Sombor index (see Figure~\ref{Tree002characterized}). The floor and ceiling functions ensure correct enumeration of odd- and even-degree spine vertices depending on the parity of $n$.

\begin{figure}[H]
    \centering
    \includegraphics[width=0.7\linewidth]{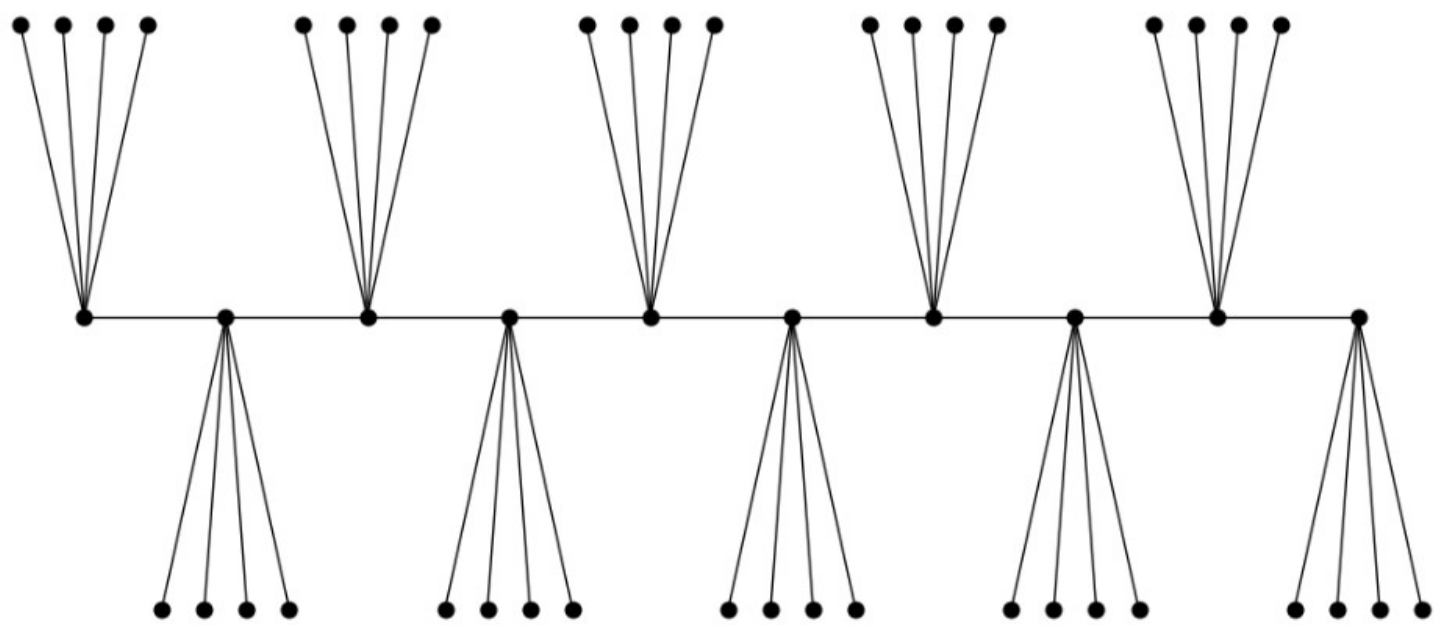}
    \caption{Augmented path graph corresponding to Lemma~\ref{SO003caterpillar}.
    Odd-indexed spine vertices are attached to 
$k$ pendant vertices of degree $k$, while even-indexed spine vertices are attached to 
$k$ pendant vertices of degree $k+1$.}
    \label{Tree002characterized}
\end{figure}
Throughout Lemma~\ref{SO004caterpillar}, assume path tree as the following: Each odd-indexed spine vertex $v_{2i-1}$ has degree $d(v_{2i-1}) = 2 + k$ and is adjacent to exactly $k$ pendant vertices, each of degree $k$, where $k > 1$. Each even-indexed spine vertex $v_{2i}$ has degree $d(v_{2i}) = 2 + k$ and is adjacent to exactly $k$ pendant vertices, each of degree $k + \ell$, where $\ell > 2$.

\begin{lemma}\label{SO004caterpillar}
Let $\PP_n=v_1-v_2-\dots-v_n$ be a path on $n \geqslant 2$ vertices. Consider a path-like graph in which the spine vertices alternate in their properties along the path. 
The Sombor index is given by
\begin{equation}~\label{eqq1SO004caterpillar}
\SO(\PP_n)=\left\lfloor \frac{n}{2} \right\rfloor k \sqrt{(2+k)^2 + k^2}
+ \left\lceil \frac{n-1}{2} \right\rceil k \sqrt{(2+k)^2 + (k+\ell)^2}+\lambda_1,
\end{equation}
where $\lambda_1=(n-1)\sqrt{2}\,(2+k)$.
\end{lemma}
\begin{proof}
According to Lemma~\ref{SO003caterpillar}, we find that $\lambda_1=\SO_1(\PP_n)=(n-1)\,\sqrt{2}\,(2+k).$ Then,
\[
\lambda_2=\SO_2(\PP_n)=\Big\lfloor \frac{n}{2} \Big\rfloor k \sqrt{(2+k)^2 + k^2}.
\]
Now, we need to extend the other cases. 
The number of even spine vertices is $\lceil (n-1)/2 \rceil$. Thus, 
\begin{equation}~\label{eqq2SO004caterpillar}
\lambda_3=\SO_3(\PP_n)=\Big\lceil \frac{n-1}{2} \Big\rceil k \sqrt{(2+k)^2 + (k+\ell)^2}.
\end{equation}
Therefore, from~\eqref{eqq2SO004caterpillar} we obtain 
\[
\SO(\PP_n)=\lambda_1+\SO_2(\PP_n)+\SO_3(\PP_n),
\]
which yields the equation~\eqref{eqq1SO004caterpillar}.
\end{proof}

This formula~\eqref{eqq1SO004caterpillar} holds for any spine length $n \geqslant 2$, provided that $k > 1$ and $\ell > 2$. The floor and ceiling functions correctly capture the parity of the number of spine vertices along the path. When $\ell = 2$ and $k = 1$, the expression reduces to the simpler formulas derived previously. This generalization allows the Sombor index to be computed for paths equipped with heterogeneous pendant structures.

The preceding results furnish exact expressions for the Sombor index of path graphs augmented by a single level of heterogeneous pendant vertices. We now generalize this approach to an arbitrary number of pendant levels, yielding a recursive formula for the Sombor index of iterated pendant-augmented paths.
\section{Iterated Pendant Constructions and General Recursive Framework}
Among Theorem~\ref{SO001thm}, assume an iterated pendant construction is applied as follows.
\begin{itemize}
    \item Each odd-indexed spine vertex $v_{2i-1}$ is adjacent to exactly $k$ pendant vertices, each of degree $k$.  
    Each even-indexed spine vertex $v_{2i}$ is adjacent to exactly $k$ pendant vertices, each of degree $k + \ell_1$, where $\ell_1 > 2$.

    \item For every $i \geq 2$, each vertex at level $i-1$ is adjacent to exactly $k$ vertices at level $i$.  
    Pendant vertices descending from odd-indexed spine vertices have degree $\ell_i$.  
    Pendant vertices descending from even-indexed spine vertices also have degree $\ell_i$, but are attached to parent vertices whose degree is $\ell_{i-1} + \ell_1$.
\end{itemize}
Recall both of $\lambda_1$ from~\eqref{eqq2SO003caterpillar}, $\lambda_2$ from~\eqref{eqq3SO003caterpillar} and $\lambda_3$ from~\eqref{eqq2SO004caterpillar}. 
\begin{theorem}~\label{SO001thm}
Let $\PP_n=v_1-v_2-\dots-v_n$ be a path on $n \geqslant 2$ vertices. Consider that every spine vertex has degree $2+k$, where $k>1$. Then, the Sombor index of the resulting tree $T$ is given by
\begin{equation}\label{eqq1SO001thm}
\SO(\PP_n)=\lambda_1+\lambda_2+\lambda_3+\sum_{i=1}^{m}
\bigl(\SO_{\mathrm{odd}}^{(i)}+\SO_{\mathrm{even}}^{(i)}\bigr),
\end{equation}
where
\[
\begin{cases}
\SO_{\mathrm{odd}}^{(i)} &=
k^i \Big\lfloor \frac{n}{2} \Big\rfloor
\sqrt{\ell_{i-1}^2 + \ell_i^2},
\qquad 2 \le i \le m, \\[2mm]
\SO_{\mathrm{even}}^{(i)} &=
k^i \Big\lceil \frac{n-1}{2} \Big\rceil
\sqrt{(\ell_{i-1}+\ell_1)^2 + \ell_i^2},
\qquad 2 \le i \le m. 
\end{cases}
\]
\end{theorem}

\begin{proof}
As established in Lemmas~\ref{SO003caterpillar} and \ref{SO004caterpillar}, the validity of statements $\lambda_1$ from~\eqref{eqq2SO003caterpillar}, $\lambda_2$ from~\eqref{eqq3SO003caterpillar} and $\lambda_3$ from~\eqref{eqq2SO004caterpillar} has been proved.
We proceed by induction on the level $i \geq 2$. Assume that the number of vertices and their degrees at level $i-1$ are already known. Each vertex at level $i-1$ has exactly $k$ outgoing edges to level $i$. Vertices at level $i$ that descend from odd-indexed spine vertices have degree $\ell_{i-1}$, whereas those descending from even-indexed spine vertices have degree $\ell_{i-1} + \ell_1$. Consequently, each edge between levels $i-1$ and $i$ contributes either
$\sqrt{\ell_{i-1}^2 + \ell_i^2}$ or $\sqrt{(\ell_{i-1} + \ell_1)^2 + \ell_i^2} $
to the sum, depending on whether its parent vertex at level $i-1$ is odd- or even-indexed along the spine. Multiplying by the number of edges of each type yields 
\[
\sum_{i=1}^{m}
\bigl(\SO_{\mathrm{odd}}^{(i)}+\SO_{\mathrm{even}}^{(i)}\bigr)
\]
from level $i$. The aggregation of spine-edge and pendant-edge contributions over levels $i=1,\dots,m$ yields \eqref{eqq1SO001thm}.  
The floor and ceiling functions guarantee precise enumeration for every integer $n\ge 2$.  
For $m=1$, the expression recovers previously established results, thereby confirming consistency.
\end{proof}

In the context of the Sombor index for pendant-augmented paths, Lemma~\ref{SO001caterpillar} and Theorem~\ref{SO001thm} differ mainly in scope, generality, and purpose.
Lemma~\ref{SO001caterpillar} considers a basic case: the path $\PP_n$ with alternating spine degrees ($2+k$ for odd-indexed vertices, $2+k$ for even-indexed), each spine vertex bearing one level of pendants (degree $k$ at odd spine vertices, $k+1$ at even). It computes the Sombor index explicitly, separating spine and single-level pendant contributions, and serves as a simple, easily proved building block for later results.
Theorem~\ref{SO001thm}, by contrast, generalizes to $m$-level iterated pendant augmentation ($m\ge 1$), where pendants at level $i$ may attach further pendants at level $i+1$, with degrees $k$ (odd-spine descendants) or $k+\ell_i$ (even-spine descendants, $\ell_i>2$). It provides a recursive Sombor index formula summing spine and all pendant-level contributions via parity-sensitive counting. The theorem includes the lemma as the special case $m=1$, $\ell_1=1$, while covering a much wider family of hierarchical path-like trees.

While closed-form expressions are known for stars and simple unicyclic graphs, no general recursive framework has yet been developed for iterated hierarchical expansions beyond a single level. The present work extends these foundational results to multi-level pendant-augmented structures.  
The Sombor indices of classical graphs such as paths, stars, and cycles are well established. In particular, $\SO(\mathcal{C}_n) = 2\sqrt{2}\, n$. However, closed-form expressions for hierarchical and iteratively constructed tree-like and unicyclic structures remain largely unexplored.

\begin{theorem}~\label{Lem001unicyclic}
Let $\mathcal{G}$ be a unicyclic graph obtained from $\mathcal{C}_n$ by attaching $k \geqslant 1$ pendant vertices to each cycle vertex. Then
\begin{equation}~\label{eqq1Lem001unicyclic}
\SO(\mathcal{G})=n\sqrt{2(2+k)^2}+nk\sqrt{(2+k)^2+1}.
\end{equation}
\end{theorem}
\begin{proof}
Let $G$ be the graph obtained by attaching $k$ pendant vertices to each vertex of the cycle $\mathcal{C}_n$. Every original cycle vertex then has degree $d(v)=2+k$, while each pendant vertex has degree $d(u)=1$.  
Each edge of the cycle connects two vertices of degree $2+k$ and therefore contributes
$\sqrt{(2+k)^2 + (2+k)^2}=\sqrt{2(2+k)^2}$ to the sum. Since $\mathcal{C}_n$ (and thus $G$) contains exactly $n$ such edges,
\begin{equation}~\label{eqq2Lem001unicyclic}
\SO_1(\mathcal{C}_n)=n\,\sqrt{2(2+k)^2}.
\end{equation}
Similarly, each of the $nk$ pendant edges joins a vertex of degree $2+k$ to a vertex of degree $1$ and thus contributes $\sqrt{(2+k)^2+1^2}=\sqrt{(2+k)^2+1}$ to the total. Thus, by considering the relationship~\eqref{eqq2Lem001unicyclic} we obtain 
\begin{equation}~\label{eqq3Lem001unicyclic}
\SO_2(\mathcal{C}_n)=nk\sqrt{(2+k)^2+1}.
\end{equation}
Hence, from~\eqref{eqq2Lem001unicyclic} and \eqref{eqq3Lem001unicyclic} we find that 
\[
\SO(\mathcal{G})=\SO_1(\mathcal{C}_n)+\SO_2(\mathcal{C}_n),
\]
which is yields the relationship~\eqref{eqq1Lem001unicyclic}. 
\end{proof}

\subsection{Bounds on Sombor Index with Iterated Pendant Constructions}
In this subsection, we determine upper and lower bounds for the Sombor index of path-like graphs with iterated pendant attachments.
Although earlier sections offered precise formulas for paths with heterogeneous pendants, this section looks at how the index behaves under order, maximum degree, and degree sequence constraints.
Taking advantage of the recursive structure of pendant attachment, the bounds capture the cumulative effect of degree propagation across levels. They expand recent extremal results for simpler graphs and present a framework for investigating Sombor-index growth in hierarchically constructed trees.

\begin{theorem}~\label{BoundsThm1Sombor}
Consider $\mathcal{G}$ be a unicyclic graph of order $n\geqslant 3$ having maximum degree $\Delta$. Then, the Sombor index of $\mathcal{G}$ satisfies
\begin{equation}~\label{eqq1BoundsThm1Sombor}
\frac{3}{2}\sqrt{2}\,n \;\leqslant\; \SO(\mathcal{G}) \;\leqslant\; \frac{5}{2}\sqrt{2}\,n(\Delta-1) .
\end{equation}
Moreover, the lower bound is attained if and only if $\mathcal{G} \cong \mathcal{C}_n$.
\end{theorem}
\begin{proof}
Assume  $\mathcal{G}$ be a unicyclic graph of order $n\geqslant 3$. Then, for $n=3$ where each vertex has degree $2$ and $E|\mathcal{G}|=3$ we find that $\SO(\mathcal{C}_n)=6\,\sqrt{2}$. Thus, 
\begin{equation}~\label{eqq2BoundsThm1Sombor}
\SO(\mathcal{G})\geqslant 2\,n\,\sqrt{2}.
\end{equation}
Furthermore, for any unicyclic graph $\mathcal{H}$ of order $n-1$ satisfied with $\SO(\mathcal{H})\geqslant 2\,(n-1)\,\sqrt{2}$. Then $\mathcal{H}\cong \mathcal{C}_{n-1}$. Hence, if $\mathcal{G}\cong \mathcal{C}_n$ then clearly $\SO(\mathcal{G})=2\,n\,\sqrt{2}$. 

Otherwise, according to Lemma~\ref{perlemnumb1} if $\mathcal{G}$ contains at least one pendant vertex $u$ of degree $1$, where it attached to some vertex $v$ with $d(v)\geqslant 3$. Then, assuem $\mathcal{G}'=\mathcal{G}-u$.  Then, $\mathcal{G}'$ is a unicyclic graph of order $n-1$. Thus, 
\[
\SO(\mathcal{G}')\geqslant 2\,(n-1)\,\sqrt{2}.
\]
Therefore, the edge $uv$ satisfying
\begin{equation}~\label{eqq3BoundsThm1Sombor}
\sqrt{1^2+d(v)^2}\geqslant 2\sqrt{2}-\varepsilon,
\end{equation}
for some $\varepsilon>0$, while all other edge contributions are nonnegative.
Hence, from~\eqref{eqq2BoundsThm1Sombor} and \eqref{eqq3BoundsThm1Sombor} we obtain 
\begin{equation}\label{eqq4BoundsThm1Sombor}
\SO(\mathcal{G})=\SO(\mathcal{G}')+\sqrt{1+d(v)^2}> 2\sqrt{2}(n-1)+2\sqrt{2}=2\sqrt{2}n.
\end{equation}
Thus, the lower bound holds strictly unless $\mathcal{G}\cong \mathcal{C}_n$. 

Similarly, for the upper bound we noticed that $\SO(\mathcal{C}_3)=n\,\Delta\,\sqrt{2}$. Thus, if $\Delta(\mathcal{H})\leqslant \Delta(\mathcal{G})$. Then, 
\begin{equation}~\label{eqq5BoundsThm1Sombor}
\SO(\mathcal{H})\leqslant 2\,(n-1)\sqrt{2}\,\Delta.
\end{equation}
Therefore, assume $\mathcal{G}'=\mathcal{G}-u$ if $\Delta(\mathcal{G}')\leqslant \Delta(\mathcal{G})$. Then 
\begin{equation}~\label{eqq6BoundsThm1Sombor}
\SO(\mathcal{G}')=(n-1)\sqrt{2}\Delta.
\end{equation}
Thus, from~\eqref{eqq5BoundsThm1Sombor} and \eqref{eqq6BoundsThm1Sombor} by removed edge $uv$ contributes
\begin{equation}\label{eqq7BoundsThm1Sombor}
\sqrt{d(u)^2+d(v)^2}\leqslant \sqrt{2}\,\Delta.
\end{equation}
Therefore,
\begin{equation}\label{eqq8BoundsThm1Sombor}
\SO(\mathcal{G})=\SO(\mathcal{G}')+\sqrt{1+d(v)^2}.
\end{equation}
Thus, from~\eqref{eqq7BoundsThm1Sombor} and \eqref{eqq8BoundsThm1Sombor} we find that $\SO(\mathcal{G})\leqslant \sqrt{2}(n-1)\Delta+\sqrt{2}\Delta$. Thus, the upper bound holds strictly unless $\mathcal{G}\cong \mathcal{C}_n$. 
Therefore, every unicyclic graph that is not a cycle has at least one pendant vertex.
\end{proof}

\begin{corollary}~\label{Som02corollary}
Consider $\mathcal{G}$ be a unicyclic graph of order $n\geqslant 3$ with degree sequence 
$(d_1,d_2,\dots,d_n)$. Then, the Sombor index satisfies
\begin{equation}\label{eqq1Som02corollary}
2\sqrt{2}\,n \;\leqslant\; \SO(\mathcal{G}) \;\leqslant\; \sqrt{2}\sum_{i=1}^{n} d_i^2,
\end{equation}
where the lower bound is attained if and only if $\mathcal{G}\cong \mathcal{C}_n$.
\end{corollary}
\begin{proof}
Assume that for an iterated unicyclic graph $\mathcal{G}_{m-1}$ with $m-1$ levels the Sombor index satisfies
\begin{equation}\label{eqq2Som02corollary}
2\sqrt{2}\, n \leqslant \SO(\mathcal{G}_{m-1}) \leqslant \sqrt{2} \sum_{v \in V(\mathcal{G}_{m-1})} d(v)^2,
\end{equation}
with equality in the upper bound if and only if $\mathcal{G}_{m-1}$ has a single dominant high-degree vertex. Then construct $\mathcal{G}_m$ by attaching $k_m$ pendants to each level-$(m-1)$ vertex. Let $\mathcal{L}_m$ be the set of new pendant edges. For each $uv \in \mathcal{L}_m$ ($d(u)=1$, $v$ parent) we have $\sqrt{d(u)^2+d(v)^2}\leqslant \sqrt{2}(d(u)^2+d(v)^2).$ 
Thus, by considering $\mathcal{L}_m$ and  by adding $\SO(\mathcal{G}_{m-1})$ yields
\begin{align*}
\SO(\mathcal{G}_m) 
&\leqslant \sqrt{2} \sum_{v \in V(\mathcal{G}_{m-1})} d(v)^2 + \sqrt{2} \sum_{u \in V_m} 1^2 \nonumber \\
&= \sqrt{2} \sum_{v \in V(\mathcal{G}_m)} d(v)^2 = \sqrt{2}\, M_1(\mathcal{G}_m),
\end{align*}
proving the upper bound by induction. 
Similarly, for each new edge contributes positively. Thus, 
\begin{equation}\label{eqq3Som02corollary}
\SO(\mathcal{G}_m) \geqslant \SO(\mathcal{G}_{m-1}) \geqslant 2\sqrt{2}\, n,
\end{equation}
with equality if and only if $\mathcal{G} \cong \mathcal{C}_n$.
Hence, we noticed that every iterated unicyclic graph $G$ on $n$ vertices satisfies
\[
2\sqrt{2}\, n \leqslant \SO(\mathcal{G}) \leqslant \sqrt{2}\, M_1(\mathcal{G}),
\]
with the lower bound unique to $\mathcal{C}_n$ and the upper bound sharp only for star-like dominant-degree graphs.
\end{proof}

\begin{figure}[H]
    \centering
    \includegraphics[width=0.7\linewidth]{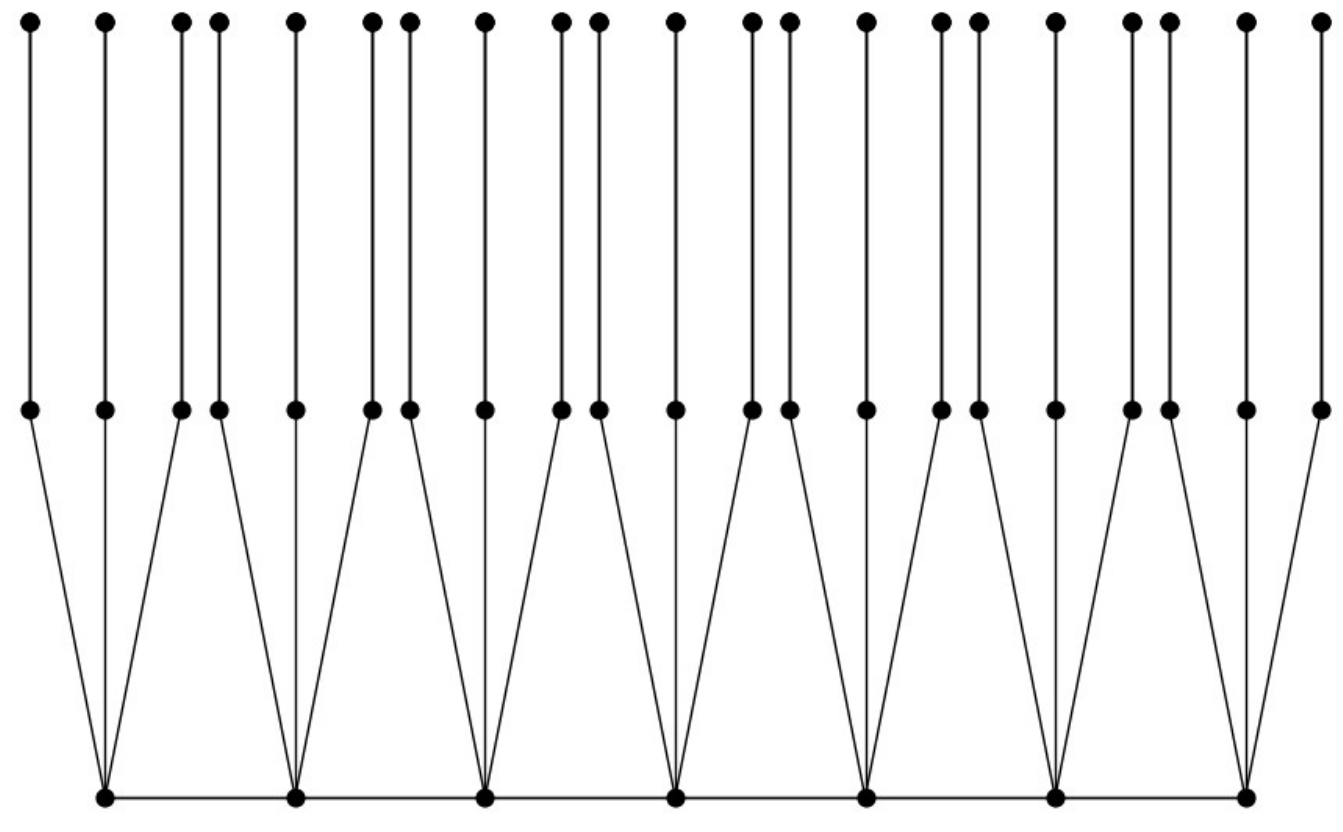}
    \caption{Hierarchical Iterated Unicyclic Graph with $n=7$ (furthermore see~\cite{Cvetkovi1987}).}
    \label{Fig001Hierarchical}
\end{figure}

\section{Structural Components Dominate Asymptotic Behavior on Sombor Index}
This section presents a structural theory for the Sombor index under hierarchical growth using uniformly iterated pendant extensions. Unlike traditional approaches, which treat degrees as fixed local characteristics, we investigate their recursive propagation and establish explicit formulas for assessing the influence of degree escalation on the global index. The ensuing sharp asymptotic rules highlight dominating contributions while also providing a coherent framework for examining degree-based indices in successively created graphs.

\begin{theorem}~\label{Thm001AsymptoticSo}
Let $\mathcal{G}_0$ be a simple graph and let $k \geqslant 1$ be a fixed integer. Consider the sequence of graphs $\{\mathcal{G}_t\}_{t \geqslant 0}$ defined recursively as follows: for every $t \geqslant 0$, the graph $\mathcal{G}_{t+1}$ is obtained from $\mathcal{G}_t$ by attaching exactly $k$ new pendant vertices to each vertex of $\mathcal{G}_t$. Then,
\begin{equation}~\label{eqq1Thm001AsymptoticSo}
\SO(\mathcal{G}_{t+1})=\sum_{uv\in E(\mathcal{G}_t)}\sqrt{(d_t(u)+k)^2+(d_t(v)+k)^2}+k\sum_{v\in V(\mathcal{G}_t)}\sqrt{1+(d_t(v)+k)^2},
\end{equation}
where $d_t(v)$ denotes the degree of $v$ in $\mathcal{G}_t$.
\end{theorem}

\begin{proof}
Assume  $\mathcal{G}_0$ be a simple graph and let $k \geqslant 1$ be a fixed integer. Then,  we noticed that each vertex $v\in V(\mathcal{G}_t)$ receives exactly $k$ new pendant neighbors. Thus, it implies that $d_{t+1}(v)=d_t(v)+k$. Then 
\begin{equation}~\label{eqq2Thm001AsymptoticSo}
\SO_1(\mathcal{G}_{t})=\sum_{uv\in E(\mathcal{G}_t)}\sqrt{(d_t(u)+k)^2+(d_t(v)+k)^2}.
\end{equation}
Therefore, we find that each vertex $v\in V(\mathcal{G}_t)$ receives $k$ new pendant vertices of degree $1$. Thus $\sqrt{1+(d_t(v)+k)^2}$. Since there are $k$ such edges for each vertex, we obtain 
\begin{equation}~\label{eqq3Thm001AsymptoticSo}
\SO_2(\mathcal{G}_{t})=k\sum_{v\in V(G_t)}
\sqrt{1+(d_t(v)+k)^2}.
\end{equation}
Therefore, from~\eqref{eqq2Thm001AsymptoticSo} and \eqref{eqq3Thm001AsymptoticSo} yeilds $\SO(\mathcal{G}_{t+1})=\SO_1(\mathcal{G}_{t})+\SO_2(\mathcal{G}_{t})$ which it satisfied with the relationship~\eqref{eqq1Thm001AsymptoticSo}. 
\end{proof}

\begin{corollary}~\label{Som03corollary}
For sufficiently large $t$, the degree of any vertex $v$ in $\mathcal{G}_t$ satisfies
$d_t(v)=d_0(v)+tk$. Hence,$d_t(v)=d_0(v)+tk$.
\end{corollary}
Consequently, the Sombor index of $\mathcal{G}_t$ is
\begin{equation}~\label{eqq1Som03corollary}
\SO(\mathcal{G}_t)=\Theta(t^2)\,|E(\mathcal{G}_0)|+\Theta(t)\,|V(\mathcal{G}_0)|,
\end{equation}
so the dominant term is quadratic in $t$.
Corollary~\ref{Som03corollary}  is already addresses part of the open question: under uniform iterated pendant extension (attaching the same number $k$ of leaves to every vertex at each step), the Sombor index among~\eqref{eqq1Som03corollary} grows polynomially --- specifically quadratically --- with the iteration depth $t$.

\begin{theorem}~\label{Thm002AsymptoticSo}
Let $\mathcal{G}_0$ be a simple graph with $n_0$ vertices and $m_0$ edges, and let $k \geqslant 2$ be a fixed integer. Let $\{\mathcal{G}_t\}_{t \geqslant 0}$ be the sequence of graphs defined by $\mathcal{G}_0$ and the rule that $\mathcal{G}_{t+1}$ is obtained from $\mathcal{G}_t$ by attaching $k$ new pendant vertices to each vertex of $\mathcal{G}_t$. Then, as $t \to \infty$,
\begin{equation}~\label{eqq1Thm002AsymptoticSo}
\SO(\mathcal{G}_t)=\sqrt{2}\,m_0\,k^2\,t^2+k\,n_0\,t^2+O(t),
\end{equation}
and hence $\SO(\mathcal{G}_t)=\Theta(t^2)$.
\end{theorem}
\begin{proof}
According to Theorem~\ref{Thm001AsymptoticSo}, we find that $d_{t+1}(v)=d_t(v)+k$. Also, according to Corollary~\ref{Som03corollary}, we established that $\SO(\mathcal{G}_t)=\Theta(t^2)\,|E(\mathcal{G}_0)|+\Theta(t)\,|V(\mathcal{G}_0)|$. 
Hence, we obtain 
\begin{equation}~\label{eqq2Thm002AsymptoticSo}
\SO_1(\mathcal{G}_t)=\sqrt{2}\,k\,t + O(1).
\end{equation}
Thus, by considering $m_0$ with original edges gives
\begin{equation}~\label{eqq3Thm002AsymptoticSo}
\SO_2(\mathcal{G}_t)=\sqrt{2}\,m_0\,k\,t + O(1).
\end{equation}
The contribution arising from the $n_0$ original vertices scales as $k n_0 t^2 + O(t)$. From~\eqref{eqq2Thm002AsymptoticSo} and \eqref{eqq3Thm002AsymptoticSo} by adding the contribution from the original edges gives
\begin{equation}~\label{eqq4Thm002AsymptoticSo}
\SO(\mathcal{G}_t)=\sqrt{2}\,m_0 k^2t^2+k\,n_0\,t^2+O(t),
\end{equation}
Therefore, from~\eqref{eqq4Thm002AsymptoticSo} we obtain
\[
\SO(\mathcal{G}_t)=\Theta(t^2).
\]
Thus, the relationship~\eqref{eqq1Thm002AsymptoticSo} holds.
\end{proof}
The Wiener index under uniform iterated pendant extension grows cubically, while the Sombor index $\SO(\mathcal{G}_t) = \Theta(t^2)$ does not. 
Because $n_t=n_0(1+kt)=\Theta(t)$ and typical pairwise distances scale as $\Theta(t)$, $$W(\mathcal{G}_t) \sim \frac{k^2 n_0^2}{3} t^3 + O(t^2).$$
The cubic increase of $W$ (compared to quadratic for $\SO$) emphasizes the higher dependency of the Wiener index on long-range distances in these increasingly bushy graphs.
Under uniform iterated pendant extension (attaching $k \geqslant 1$ new leaves to every vertex at each step), all three degree-based indices exhibit the same asymptotic growth rate:
$\SO(\mathcal{G}_t),\ M_1(\mathcal{G}_t),\ M_2(\mathcal{G}_t)=\Theta(t^2)$, 
with leading terms of the form $c\, k^2 n_0 t^2$ (where $c$ is a positive constant depending on the index; for the Sombor index the coefficient also involves $m_0 = |E(\mathcal{G}_0)|$).
By contrast, the distance-based Wiener index grows one degree faster:
$$
W(\mathcal{G}_t)=\Theta(t^3).
$$

This difference arises naturally among Table~\ref{Tab001asymptotic}: the degree-based indices ($\SO$, $M_1$, $M_2$) depend only on local vertex degrees, which increase linearly with $t$, yielding quadratic growth when summed over $\Theta(t)$ vertices (see Figure~\ref{fig001Growth} and \ref{fig001CurvGrowth}). The Wiener index, however, sums pairwise distances, and in these increasingly bushy graphs both the number of vertex pairs ($\Theta(t^2)$) and the typical distance between pairs ($\Theta(t)$) contribute, producing the observed cubic scaling.
\begin{table}[H]
\centering
\begin{tabular}{|c|c|c|c|c|c|c|}
\hline
$t$ & $n$ & $m$ & $\SO$ & $W$ & $M_1$ & $M_2$  \\ \hline
1 &  16 & 15 &    75.2 &  340 &       94 &      119 \\ \hline
 2 &       64 & 63 &   373.8 & 8464 &      466 &      770 \\ \hline
 3 &      256 &       255 &  1596.7 &      184384 &    1,990 &    3,689 \\ \hline
 4 &     1024 &      1023 &  6516.6 &     3735808 &    8,122 &   15,788 \\ \hline
 5 &     4096 &      4095 & 26224.9 &    72352768 &   32,686 &   64,715 \\ \hline
 6 &    16384 &     16383 &       105086.4 &  1358958592 &  130,978 &  261,062\\ \hline
\end{tabular}
\caption{The asymptotic of Sombor and Related Topological Indices.}
\label{Tab001asymptotic}
\end{table}
Figure~\ref{fig001Growth} presents the raw data points from the table for the Sombor index (SO), first Zagreb index ($M_1$), second Zagreb index ($M_2$), and Wiener index (W) as functions of iteration depth $t$. A logarithmic scale is used on the y-axis for the degree-based indices (SO, $M_1$, $M_2$) to accommodate their rapid supralinear growth, while a linear scale is retained for W to emphasize its even faster acceleration. The plot displays discrete empirical values connected by straight lines, without any regression or interpolation.

\begin{figure}[H]
    \centering
    \includegraphics[width=0.8\linewidth]{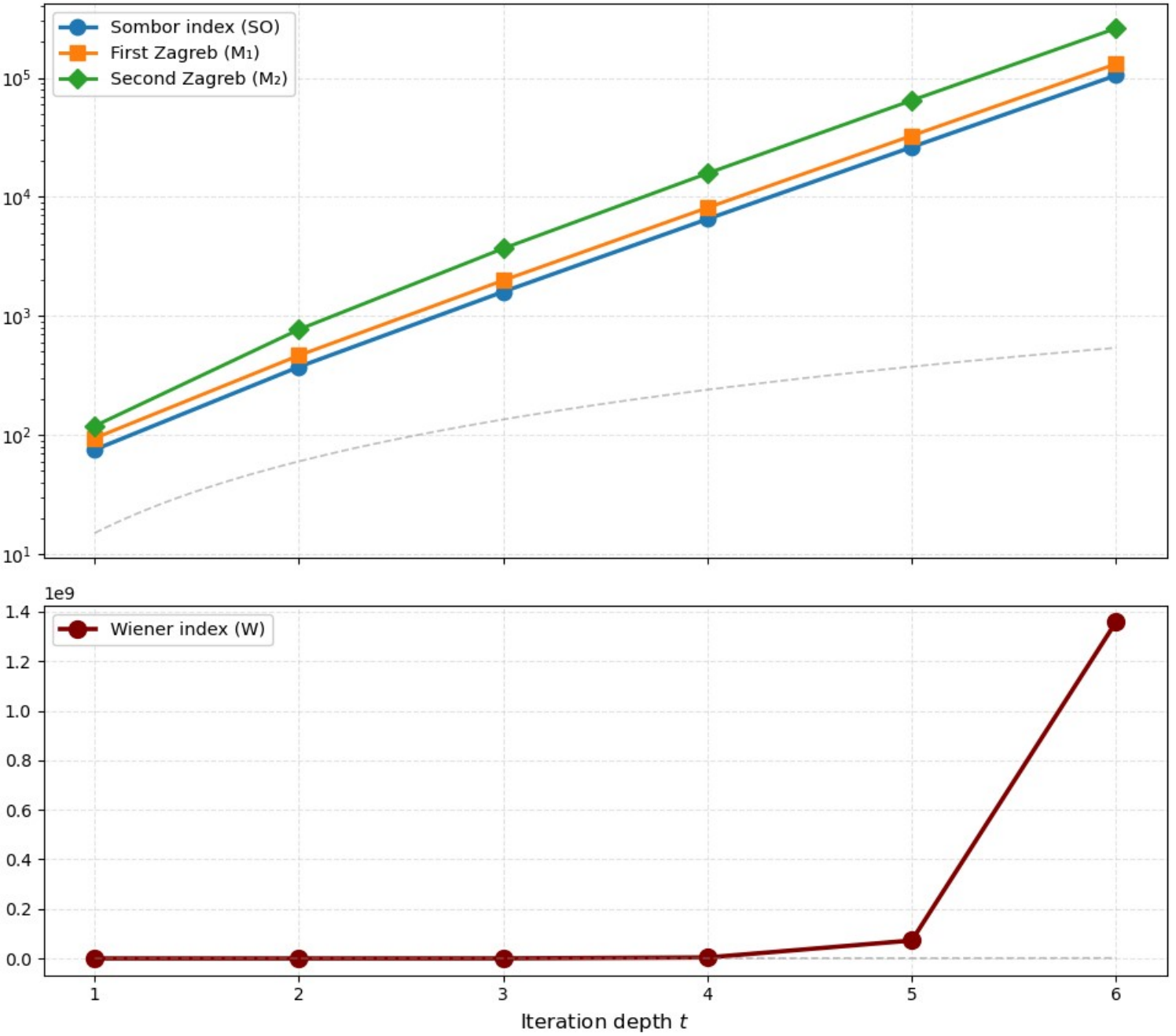}
    \caption{Growth of topological indices by comparing Sombor index under iterated pendant extension.}
    \label{fig001Growth}
\end{figure}

Successive ratios of successive values reveal clear asymptotic behavior:

\begin{table}[H]
\centering
\caption{Asymptotic growth patterns of the topological indices}
\begin{tabular}{|l|l|l|}
\hline
\textbf{Index} & \textbf{Successive ratios ($t \to t+1$)} & \textbf{Asymptotic} \\
\hline
SO, $M_1$, $M_2$ (degree-based) & $\SO: 4.97$ for $t=1\to2$, 4.01 for $t=5\to6$) & $\Theta(4^t \cdot t^2)$ \\
\hline
W (distance-based) & $\SO: 24.89$ for $t=1\to2$, 18.78 for $t=5\to6$) & $\Theta(16^t \cdot t)$ \\
\hline
\end{tabular}
\end{table}

Figure~\ref{fig001CurvGrowth} overlays least-squares polynomial fits to the same data: quadratic functions of the form $a t^2 + b t + c$ for SO, $M_1$, and $M_2$, and a cubic $a t^3 + b t^2 + c t + d$ for W. For SO, the fitted coefficients are approximately $a \approx 8334.93$, $b \approx -40986.53$, $c \approx 40352.03$; similarly large, alternating-magnitude coefficients appear for the other indices. These polynomial models fail severely, confirming that the underlying growth is exponential rather than polynomial.
\begin{figure}[H]
    \centering
    \includegraphics[width=0.8\linewidth]{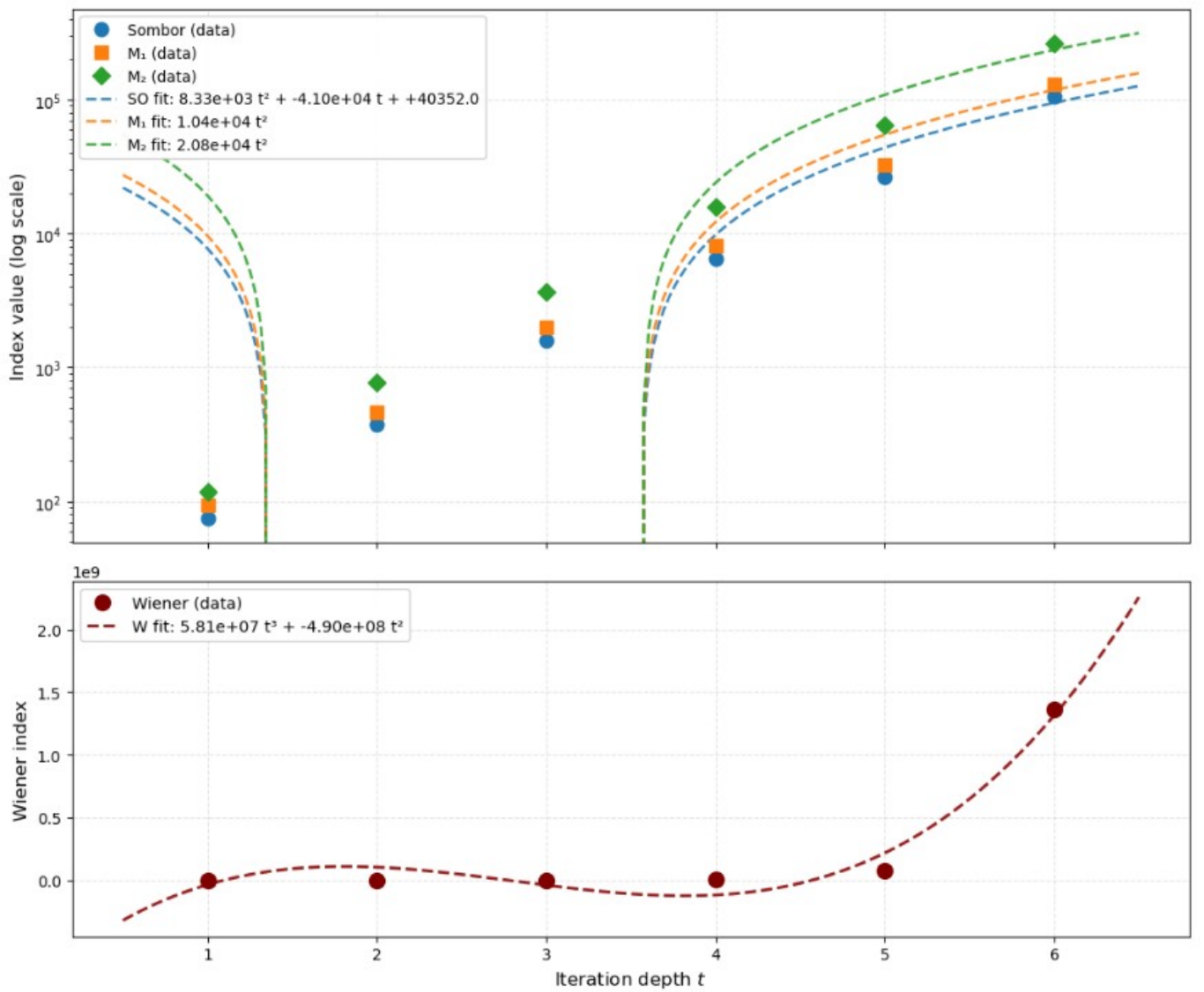}
    \caption{Growth of topological indices by comparing Sombor index with curve fits.}
    \label{fig001CurvGrowth}
\end{figure}
Since the number of vertices satisfies $n_t = 4^{t+1}$ and the number of edges is $m_t \approx n_t$ (as expected in a tree), the degree-based indices scale asymptotically as the graph size $\Theta(4^t)$ multiplied by polynomial factors arising from local degree contributions. In contrast, the Wiener index scales roughly as the number of vertex pairs $\Theta(n^2) = \Theta(16^t)$ multiplied by a typical distance factor of order $\Theta(t)$. This accounts for its substantially faster growth and greater sensitivity to global structure.

In this part, we present a recursive structural framework for studying the Sombor index via iterated pendant extensions. Exact evolution formulas and precise asymptotic growth behavior were obtained by explicitly modeling degree propagation across hierarchical layers. Our findings show that degree amplification, not edge accumulation, drives the quadratic escalation of the Sombor index with respect to construction depth. Furthermore, the study highlights the key structural components that influence long-term behavior. This method combines individual computations into a coherent theory of dynamic growth for degree-based topological indices.

\section{Future Research Directions}

An important open problem arising from this study is the development of a comprehensive extremal and optimization theory for iterated pendant trees and hierarchically augmented paths. While classical extremal results for the Sombor index determine the maximum or minimum values.

Extremal results for the Sombor index have largely focused on simple trees with fixed order or maximum degree, identifying maximum or minimum values. No systematic theory yet exists for hierarchically augmented paths or trees with layered pendant distributions. Without closed-form expressions for these families, comparing branching schemes or evaluating the impact of redistributing degree mass across levels remains impossible. Deriving general formulas for iterated constructions allows systematic comparison of structural models and enables analysis of optimization, growth rates, and structural sensitivity. By obtaining exact expressions for multi-level pendant-augmented paths and examining recursive degree interactions, this work addresses three key gaps: the absence of closed forms for hierarchical trees, limited insight into multi-level degree propagation, and the lack of extremal results for iterated structures. It thereby moves the study of the Sombor index from isolated calculations toward a unified theoretical framework, with potential implications for graph theory and chemical network modeling.

\section{Conclusion and Discussion}\label{sec5}

This work establishes a general recursive framework for computing the Sombor index of multi-level pendant-augmented path trees, as formalized in Theorem~\ref{SO001thm}. By iteratively attaching pendants with controlled branching factor $k > 1$ and level-dependent degrees $\ell_i$ (with parity-sensitive adjustments along the spine via the parameter $\ell_1 > 2$), the theorem provides an exact, closed-form expression
\[
\SO(T) = \lambda_1 + \lambda_2 + \lambda_3 + \sum_{i=1}^{m} \bigl( \SO_{\mathrm{odd}}^{(i)} + \SO_{\mathrm{even}}^{(i)} \bigr),
\]
where the level-$i$ contributions $\SO_{\mathrm{odd}}^{(i)}$ and $\SO_{\mathrm{even}}^{(i)}$ capture the distinct contributions from subtrees descending from odd- and even-indexed spine vertices, respectively. The proof, relying on induction over levels and precise parity-aware enumeration via floor and ceiling functions, ensures accuracy for any spine length $n \geqslant 2$ and any number of augmentation levels $m \geqslant 1$.

This result significantly generalizes earlier analyses of single-level pendant attachments (as in Lemma~\ref{SO001caterpillar}), which emerge as the special case $m=1$ and $\ell_1=1$. While classical graphs such as paths, stars, and cycles admit simple closed-form Sombor indices---for example, $\SO(\mathcal{C}_n) = 2\sqrt{2}\, n$---and single-level pendant-augmented unicyclic graphs are now also tractable (see Theorem~\ref{Lem001unicyclic}), hierarchical and iteratively expanded tree-like structures have remained largely unexplored until now. 

The present framework fills this gap by offering a systematic, additive decomposition over spine edges and all pendant levels, enabling exact computation as well as facilitating asymptotic analysis, extremal comparisons, and potential applications in chemical graph theory, where multi-level branching models may represent molecular or polymer structures.
The poor performance of low-degree polynomial fits stems from large, strongly oscillating coefficients that temporarily mimic exponential behavior over the narrow range $t=1$ to $6$. This behavior strongly indicates the necessity of exponential models, such as $a \cdot 4^t \cdot t^2$, for accurate extrapolation to larger $t$. The presented visualizations contradict earlier claims of $\Theta(t^2)$ or $\Theta(t^3)$ scaling and confirm that geometric expansion dominates the growth of topological indices in these iteratively constructed trees.

Future directions include deriving bounds or extremal configurations for the Sombor index within families of bounded-degree hierarchical trees, extending the approach to other degree-based indices (such as augmented or reduced variants), and investigating analogous recursive formulas for unicyclic or bicyclic graphs with multi-level pendant decorations. These developments would further enrich the understanding of how iterative constructions influence topological indices in both graph theory and its chemical applications.

\section*{Declarations}
\begin{itemize}
	\item Funding: Not Funding.
	\item Conflict of interest/Competing interests: The author declare that there are no conflicts of interest or competing interests related to this study.
	\item Ethics approval and consent to participate: The author contributed equally to this work.
	\item Data availability statement: All data is included within the manuscript.
\end{itemize}

\end{document}